\documentclass[12pt]{amsart}

\usepackage{amscd}
\usepackage{amsmath, amssymb}
\usepackage{amsfonts}
\usepackage{mathrsfs}
\newcommand{\de}{\partial}

\newcommand{\ddb}{\partial \ov{\partial}}

\newcommand{\ddbar}{\sqrt{-1} \partial \overline{\partial}}

\newcommand{\ov}[1]{\overline{#1}}

\newcommand{\ti}[1]{\tilde{#1}}

\newcommand{\vol}{\mathrm{Vol}}

\newcommand{\ve}{\varepsilon}

\renewcommand{\leq}{\leqslant}
\renewcommand{\geq}{\geqslant}

\newcommand{\be}{\begin{equation}}
\newcommand{\ee}{\end{equation}}

\begin{document}
\newcounter{remark}
\newcounter{theor}
\setcounter{remark}{0}
\setcounter{theor}{1}
\newtheorem{claim}{Claim}
\newtheorem{theorem}{Theorem}[section]
\newtheorem{lemma}[theorem]{Lemma}
\newtheorem{corollary}[theorem]{Corollary}
\newtheorem{proposition}[theorem]{Proposition}
\newtheorem{question}{question}[section]
\newtheorem{defn}{Definition}[theor]
\newtheorem{case}{Case}[theor]
\numberwithin{equation}{section}

\title[Scalar V-soliton equation]{Scalar $V$-soliton equation and K\"ahler-Ricci flow on symplectic quotients}

\author{Chang Li}
\address{School of Mathematical Sciences, Peking University, Beijing, P.R.China, 100871}
\email{chang\_li@pku.edu.cn}

\begin{abstract}
In this paper, we consider the $V$-soliton equation which is a degenerate fully nonlinear equation introduced by La Nave and Tian in their work on K\"ahler-Ricci flow on symplectic quotients. One can apply the interpretation to study finite time singularities of the K\"ahler-Ricci flow. As in the case of K\"ahler-Einstein metrics, we can also reduce the $V$-soliton equation to a scalar equation on K\"ahler potentials, which is of Monge-Amp\`ere type. We formulate some preliminary estimates for such a scalar equation on a compact K\"ahler manifold $M$.
\end{abstract}
\maketitle

\section{Introduction}
The Ricci flow was introduced by R. Hamilton in 1982 \cite{H1}. Hamilton conjectured that the flow would break the manifold into pieces at the singular times, and initiated a program of ¡®Ricci flow with surgeries¡¯  \cite{H2}. The method of Ricci flow with surgery has become a very powerful tool to study the topology and geometric structures of Riemannian manifolds. In \cite{P1} \cite{P2} \cite{P3}, Perelman's ground-breaking work refined Hamilton's surgery process and proved Thurston's geometrization conjecture for 3-manifolds.

There has been much interest in studying the phenomenon in the case of K\"{a}hler manifolds. Let $(X,\omega)$ be an $m$-dimensional compact K\"{a}hler manifold. We denote a K\"{a}hler metric by its K\"{a}hler form $\omega$. Because the Ricci flow preserves the K\"{a}hler condition, it leads to the K\"{a}hler-Ricci flow as follows:
\begin{equation} \label{Kahler Ricci flow}
\frac{\partial}{\partial t} \omega=-\operatorname{Ric}(\omega), \quad\left.\omega\right|_{t=0}=\omega_{0},
\end{equation}
where $\omega_{0}$ is any given K\"{a}hler metric and $\mathrm{Ric}(\omega)$ denotes the Ricci form of $\omega$.
As long as the flow exists, the K\"{a}hler class of $\omega(t)$ is given by
\begin{equation}
[\omega(t)]=\left[\omega_{0}\right]+t c_{1}\left(K_{X}\right)>0,
\end{equation}
where $K_{X}$ is the canonical bundle of $X$. The Ricci flow always has a solution for $t$ small. In \cite{TZha}, Tian-Zhang gave a sharp local existence for \eqref{Kahler Ricci flow}. For any initial K\"ahler metric $\omega_{0}$, the flow \eqref{Kahler Ricci flow} has a maximal solution $\omega_{t}$ on $X \times [0, T_{max})$, where
\begin{equation}
T_{\max }=\sup \left\{t \in \mathbb{R} | \left[\omega_{0}\right]+t [K_{X}]>0\right\}.
\end{equation}

The K\"ahler-Ricci flow has been studied extensively and become a powerful tool in K\"ahler geometry. Cao \cite{Cao} studied the K\"ahler-Ricci flow and proved that the K\"{a}hler-Ricci flow always converges exponentially fast to a K\"{a}hler-Einstein metric if the first Chern class is negative or zero. Using this, Cao gave an alternative proof of the Calabi conjecture which was first proved by Yau \cite{Yau78}.
If the first Chern class is positive and $T$ is finite, the behavior of $g(t)$ as $t$ tends to $T$ has been extensively studied (see \cite{P4, TZhu1, PSSW, SeT, CS, TiZ1, CW, B} etc.). In the cases that the K\"{a}hler manifolds do not admit definite or vanishing first Chern class, the flow will in general develop singularities. Song-Tian \cite{ST1, ST2, ST3} initiated an Analytic Minimal Model Program through Ricci flow which is parallel to Mori's birational minimal model program (see \cite{BCHM, HM}). The study of formation of singularities along the K\"{a}hler-Ricci flow is a crucial problem in this program. In \cite{ST3}, it was conjectured that the K\"ahler-Ricci flow will either deform a projective algebraic variety to its minimal model via finitely many divisorial contractions and flips in the Gromov-Hausdorff sense, and then converge (after normalization) to a generalized K\"ahler-Einstein metric on its canonical model, or collapse in finite time (see also \cite{Tian1}).
This extended earlier conjectures of Feldman-Ilmanen-Knopf \cite{FIK}, which were established in \cite{SW1, CT}.
These surgeries were proved to be continuous in the Gromov-Hausdorff topology for K\"ahler surfaces in \cite{SW2, SW3}. Higher dimensional metric surgeries via the K\"ahler-Ricci flow are constructed for certain families of projective manifolds in \cite{S-Y, S1}. For more details of the behavior of the K\"{a}hler-Ricci flow with finite time singularity, see \cite{Zhang1, SSW, Fong1, Song1, Zhang2, Fong2, FZ, TZ}.

In \cite{LT}, La Nave and Tian studied the K\"{a}hler-Ricci flow through singularities in the frame work of K\"{a}hler quotients by identifying the K\"{a}hler-Ricci flow with a degenerate fully nonlinear equation on the resolution. The authors proposed an interpretation of the K\"{a}hler-Ricci flow on a manifold $X$ as an exact elliptic equation of Einstein type on a manifold $M$ of which $X$ is one of the (K\"{a}hler)
symplectic reductions via a (non-trivial) torus action. Note that a large class of birational transformations can be constructed through symplectic quotients. Then their interpretation may reduce studying singularities of the
K\"{a}hler-Ricci flow on quotients to studying an elliptic problem on $M$.

More precisely, let $(M, g)$ be a K\"{a}hler manifold of dimension $n\geq2$ which admits a Hamiltonian $S^{1}$-action by holomorphic automorphisms and let $V$ be the vector field generating such an action.
Let $\mu:M\rightarrow \mathbb{R}$ be a moment map or a Hamiltonian function for this action, i.e.
\begin{equation}
i_{V} \omega_{g}=\sqrt{-1} d \mu.
\end{equation}
We will call $\tilde{g}$ a V-soliton metric if it is K\"{a}hler and satisfies:
\begin{equation}\label{V-soliton metric}
\operatorname{Ric}(\tilde{g})+\frac{\sqrt{-1}}{2} \partial \overline{\partial}\left(\log \left(|V|_{\tilde{g}}^{2}\right)+f \cdot \mu\right)=\lambda \omega_{\tilde{g}},
\end{equation}
where $f$ is a function on $\mathbb{R}$. This can be regarded as a generalization of K\"{a}hler-Einstein metrics or K\"{a}hler-Ricci solitons.

Let $S(\mu) \subset \mathbb{R}$ be the set of singular values. It is closed, so the set $R(\mu) :=\mathbb{R} \backslash S(\mu)$ of regular values is a disjoint union of open intervals. For any regular value $\tau$, $\mu^{-1}(\tau)$ is smooth and $S^{1}$ acts freely on it, so the symplectic quotient $X_{\tau}=\mu^{-1}(\tau) / S^{1}$ is a smooth manifold. Moreover, since $S^{1}$ acts on $M$ by holomorphic automorphisms, $X_{\tau}$ has an natural complex structure $J_{\tau}$ induced from the one on $M$. The restriction of $\tilde{g}$ to $\mu^{-1}(\tau)$ descends to a K\"{a}hler metric $\tilde{g} _ \tau$ on $X_{\tau}$. Hence, $\left(X_{\tau}, \tilde{g}_{\tau}\right)$ is a K\"{a}hler manifold. If $\tau \in\left[\tau_{0}, \tau_{1}\right] \subset R(\mu)$, then all these $X_{\tau}$ are biholomorphic to each other, we may assume they all biholomorphic to a fixed complex manifold $X$. However, induced K\"{a}hler metrics $\tilde{g}_{\tau}$ depend on $\tau$. We consider  K\"{a}hler metrics which are invariant under the $S^{1}$-action. The main theorem in \cite{LT} is as follows:

\begin{theorem}
If $\omega_{\tilde{g}}$ is a V-soliton metric on $M$ and invariant under the $S^{1}$-action, then for $\tau \in\left[\tau_{0}, \tau_{1}\right]$, $\omega_{\tilde{g}_{\tau(t)}}$ on $X = X_{\tau}$ is a solution of \eqref{Kahler Ricci flow}, where $\tau (t) = ct$ for an appropriate constant $c > 0$. Conversely, a solution $\omega_{t}$ of \eqref{Kahler Ricci flow} can be lifted to be a V-soliton on an open subset $U$ of a certain $M$ with a $S^{1}$-action and a Hamiltonian $\mu$ such that $\mu(U) \subset R(\mu)$.
\end{theorem}

\noindent This interpretation can be also extended to any symplectic quotients by more general groups. An holomorphic Hamiltonian action of a Lie group $G$ on a manifold $M$ comes with a moment map $\mu : M \rightarrow \mathcal{G}^{*}$ : For every coadjoint orbit in the dual Lie algebra of $G$, $\tau \subset \mathcal{G}^{*}$, there is a K\"{a}hler quotient $X_{\tau} :=\mu^{-1}(\tau) / G$, as above, we may have an elliptic equation on $M$ whose solutions can descend to solutions of the K\"{a}hler-Ricci flow on $X_{\tau}$.

Similarly to the case of K\"{a}hler-Einstein metrics, we can reduce \eqref{V-soliton metric} to a scalar equation on K\"{a}hler potentials, which is of Monge-Amp\`{e}re type. To be more explicit, we choose $g$ such that $c_{1} (M)$ coincides $\lambda\left[\omega_{g}\right]$. Then we can write $\omega_{\tilde{g}}$ as
$\omega_{\tilde{g}}=\omega_{g}+\frac{\sqrt{-1}}{2} \partial \overline{\partial} u$.
Consider the following
\begin{equation}\label{scalar V-soliton equation}
\left(\omega_{g}+\frac{\sqrt{-1}}{2} \partial \overline{\partial} u\right)^{n}=|V|_{\tilde{g}} e^{F-\lambda u} \omega^{n},
\end{equation}
where $F$ is a given smooth function satisfying
\begin{equation*}
\int_{M}\left(|V|_{g}^{2} e^{F}-1\right) \omega_{g}^{n}=0.
\end{equation*}
We call \eqref{scalar V-soliton equation} a scalar $V$-soliton equation.

One motivation for studying \eqref{scalar V-soliton equation} comes from establishing the existence of $V$-soliton metric.
If $F$ satisfies
\begin{equation*}
\operatorname{Ric}\left(g\right)-\lambda \omega_{g}=\sqrt{-1} \partial \overline{\partial}(F-f),
\end{equation*}
then $\tilde{g}$ is a $V$-soliton metric, i.e. $\tilde{g}$ satisfies \eqref{V-soliton metric}.

Since $V$ is generated by $S^{1}$-action, we obtain that $V$ is real. We can rewrite \eqref{scalar V-soliton equation} in a slightly different way: Let $Z=JV+\sqrt{-1} V$. Then $Z$ is a holomorphic vector field on $M$ and $|Z|_{\tilde{g}}^{2}=2|V|_{\tilde{g}}^{2}$.
Thus \eqref{scalar V-soliton equation} becomes
\begin{equation} \label{scalar V-soliton equation, Z}
\left(\omega_{g}+\frac{\sqrt{-1}}{2} \partial \overline{\partial} u\right)^{n}=\frac{1}{2}|Z|_{\tilde{g}} e^{F-\lambda u} \omega_{g}^{n}.
\end{equation}

We may use the perturbation method to solve \eqref{scalar V-soliton equation, Z}. Consider
\begin{equation}  \label{scalar V-soliton equation, Z,perturbation}
(\omega_g+\frac{\sqrt{-1}}{2} \partial \overline{\partial} u)^{n}=(|Z|_{\tilde{g}}^{2}+\ve)e^{F+c_{\ve}-\lambda u}\omega^{n},
\end{equation}
where $0 <\ve <1$, and $C_{\ve}$ is chosen such that
\begin{equation*}
\int_{M}\left(\left(\ve+|Z|_{g}^{2}\right) e^{F_{\ve}}-1\right) \omega^{n}=0,
\end{equation*}
where $F_{\ve}=F+c_{\ve}$.

Our main goal in this paper is to develop some a $priori$ estimates of solutions for this scalar $V$-soliton equation. For simplicity, we assume that $M$ is compact. Here we consider \eqref{scalar V-soliton equation, Z} only when $\lambda \leq 0$.

For the zero order estimate, it was proved by La Nave-Tian \cite{LT} when $\lambda=0$. When $\lambda=-1$, we adapt an approach of Tosatti-Weinkove \cite{TW}. We prove
\begin{proposition}\label{zero estimate negative lambda}
When $\lambda=-1$. Assume that there is a uniform constant $C_{0}$ such that $\inf_{M} u \leq C_{0}$ for any $S^{1}$-invariant solution $u$ for \eqref{scalar V-soliton equation, Z,perturbation}, then there is a uniform constant $C$ depending only on $Z, F, (M,\omega), C_0$ such that
\begin{equation}\label{c0 estimate}
\sup _{M}|u| \leq C.
\end{equation}
In particular, if $|\operatorname{div}(Z)|^2-\operatorname{Ric}(Z,Z) \geq 0$, then there is a uniform constant $C_{0}$ such that
\begin{equation}
\inf_{M} u \leq C_{0}
\end{equation}
for any $S^{1}$-invariant solution $u$. This implies the uniform estimate \eqref{c0 estimate}. Note that the condition $\operatorname{Ric} \leq 0$ is a special case.
\end{proposition}

For the second order estimates, we prove the following:

\begin{proposition}\label{Laplacian estimate}
There exists a constant $C$ depending only on $\| u\|_{C^{0}}$, $Z$, $F$ and $(M,\omega)$ such that for any solution $u$ for \eqref{scalar V-soliton equation, Z,perturbation}, we have
\begin{equation}
\sup_{M}\Delta u \leq C\sup_{M}|Z|_{\tilde{g}}^{2}+C.
\end{equation}
\end{proposition}

\noindent In some special cases, we can get the Laplacian estimate.

\begin{proposition}\label{Laplacian estimate Positive Curvature}
When $(M,\omega)$ has positive holomorphic bisectional curvature, there exists a constant $C$ depending only on $\| u\|_{C^{0}}$, $Z$, $F$ and $(M,\omega)$ such that for any solution $u$ for \eqref{scalar V-soliton equation, Z,perturbation}, we have
\begin{equation}
\sup_{M}\Delta u \leq C.
\end{equation}
\end{proposition}

\begin{proposition}\label{Laplacian estimate negative lambda}
When $\lambda =-1$, there exists a constant C depending only on $Z$, $F$ and $(M,\omega)$ such that for any solution $u$ for \eqref{scalar V-soliton equation, Z,perturbation}, we have
\begin{equation}
\sup_{M}|Z|_{\tilde{g}}^{2} \leq C.
\end{equation}
Hence, we have
\begin{equation}
\sup_{M}\Delta u \leq C',
\end{equation}
here $C'$ depending on $\| u\|_{C^{0}}$, $Z$, $F$ and $(M,\omega)$.
\end{proposition}

The scalar V-soliton equation \eqref{scalar V-soliton equation} has been much less studied so far.
In \cite{G-L}, Guan-Li studied the following Monge-Amp\`ere type equation:
\begin{equation}
\left\{\begin{array}{ll}
{(\chi+\frac{\sqrt{-1}}{2} \partial \overline{\partial} u)^{n}=\psi (\chi+\frac{\sqrt{-1}}{2} \partial \overline{\partial} u) \wedge \omega^{n-1}} & {\text { on } M}, \\
{\chi+\frac{\sqrt{-1}}{2} \partial \overline{\partial} u>0} & {\text { on } M}, \\
{u=\varphi} & {\text { on } \partial M},
\end{array}\right.
\end{equation}
where $\chi$ is a smooth real (1,1)-form on $M$ and $\psi \in C^{\infty}(M)$ is a strictly positive function. The authors solved this Dirichlet problem on compact Hermitian manifolds admitting a subsolution.
One of the motivation which was important in leading the authors to this problem is the scalar V-soliton equation \eqref{scalar V-soliton equation}. In local coordinates, the equation above can be written as
\begin{equation}
\operatorname{det} (g_{i \bar{j}}+u_{i \bar{j}})=\frac{\psi}{n} (\operatorname{tr} \chi+\Delta u) \operatorname{det} g_{i \bar{j}},
\end{equation}
which is a similar form as \eqref{scalar V-soliton equation}.

The rest of the paper is organized as follows. In Section 2, we introduce some basic results, and notations, in K\"ahler geometry and symplectic quotients. We prove Proposition \ref{zero estimate negative lambda} in Section 3 and Proposition \ref{Laplacian estimate}--Proposition \ref{Laplacian estimate negative lambda} in Section 4.

\textbf{Acknowledgements}:
The author would like to thank his advisor Prof. Gang Tian for leading him to study this problem, sharing his unpublished notes with Gabriele La Nave, constant encouragement and support.
The author would also like to thank Prof. Jian Song for his helpful comments and suggestions.
Partial work was done while the author was visiting the Department of Mathematics at Rutgers University, supported by Graduate School of Peking University. The author would like to thank the Department of Mathematics at Rutgers University for their warm hospitality.
The author would also like to thank Jianchun Chu for many helpful discussions and very careful reading of the preprint.

\section{Preliminary}
In this section, we collect some preliminary results.
\subsection{Basic results and notation in K\"ahler geometry}
Let $(M, g, J)$ be a K\"ahler manifold of dimension $n$. Here $J$ is the induced almost complex structure on $M$ and $g$ is the Hermitian metric compatible with $J$. We denote $\omega=\omega_{g}$ the K\"ahler form of $g$. Recall that $g$ is K\"ahler if $\omega$ is closed.

The complexified tangent bundle $T^{\mathbb{C}} M=T M \otimes \mathbb{C}$ has a natural splitting
\begin{equation*}
T^{\mathbb{C}} M=T^{1,0} M+T^{0,1} M,
\end{equation*}
where $T^{1,0} M$ and $T^{0,1} M$ are the eigenspaces of $J$, corresponding to eigenvalues $\sqrt{-1}$ and $-\sqrt{-1}$ respectively. The metric $g$ is obviously extended $\mathbb{C}$-linearly to $T^{\mathbb{C}} M$.
Let $\nabla$ be the Chern connection of $g$. Then it satisfies
\begin{equation*}
\nabla_{W}(g(X, Y))=g\left(\nabla_{W} X, Y\right)+g\left(X, \nabla_{W} Y\right).
\end{equation*}
The curvature tensor $R$ of $\nabla$ is defined by
\begin{equation*}
\begin{split}
R(X, Y) U={}& \nabla_{X} \nabla_{Y} U-\nabla_{Y} \nabla_{X} U-\nabla_{[X, Y]} U,\\[2mm]
R(X, Y, U, W)={}& g\left(R(X,Y)W, U\right).
\end{split}
\end{equation*}
Because $J$ is parallel, we see that
\begin{equation*}
\begin{split}
R(X, Y) J U={}& J R(X, Y) U,\\[2mm]
R(X,Y,JU,JW)={}& R(X,Y,U,W).
\end{split}
\end{equation*}
Given $X, Y \in T_{p}^{1,0} M \backslash\{0\}$, the holomorphic bisectional curvature of $\omega$ at $p$, determined by $X, Y$ is defined to be
\begin{equation*}
BK(X, Y) :=\frac{R(X, \overline{X}, Y, \overline{Y})}{|X|_{\omega}^{2}|Y|_{\omega}^{2}}.
\end{equation*}

Under the local complex coordinates $(z^{1}, \cdots, z^{n})$, we can write
\begin{equation*}
\omega=\frac{\sqrt{-1}}{2}g_{i \overline{j}} d z^{i} \wedge d \overline{z}^{j},
\end{equation*}
where $g_{i \overline{j}}=g(\frac{\partial}{\partial z^{i}}, \frac{\partial}{\partial \overline{z}^{j}})$.
Then the connection $\nabla$ is given by
\begin{equation*}
\left\{\Gamma_{j k}^{i}=g^{i \overline{l}} \frac{\partial g_{j \overline{l}}}{\partial z^{k}}\right\}.
\end{equation*}
Hence, the curvature tensor $R$ is represented as
\begin{equation*}
\begin{split}
R_{i \overline{j} k \overline{l}}
={}& R\left(\frac{\partial}{\partial z^{i}}, \frac{\partial}{\partial \overline{z}^{j}}, \frac{\partial}{\partial z^{k}}, \frac{\partial}{\partial \overline{z}^{l}}\right)\\[2mm]
={}& -\frac{\partial^{2} g_{i \overline{j}}}{\partial z^{k} \partial \overline{z}^{l}}+g^{s \overline{t}} \frac{\partial g_{s \overline{j}}}{\partial z^{k}} \frac{\partial g_{i \overline{t}}}{\partial \overline{z}^{l}}.
\end{split}
\end{equation*}
We define the Ricci curvature $\operatorname{Ric}(\omega)=\frac{\sqrt{-1}}{2}R_{k \ov l} dz^{k} \wedge d \overline{z}^{l}$ to be the trace of $R$, so we get
\begin{equation*}
R_{k \overline{l}}=g^{i \overline{j}} R_{i \overline{j} k \overline{l}}=-\frac{\partial^{2}}{\partial z^{k} \partial \overline{z}^{l}}\left(\log \operatorname{det} g_{i \overline{j}}\right).
\end{equation*}
It is the same as the one in Riemannian geometry.
For a function $f \in C^{2}(M)$, $\partial \overline{\partial} f$ is given in local coordinates by
\begin{equation*}
\partial \overline{\partial} f=\frac{\partial^{2} f}{\partial z^{i} \partial \overline{z}^{j}} d z^{i} \wedge d \overline{z}^{j}.
\end{equation*}
We define the canonical Laplacian of $f$ respect to the Chern connection by
\begin{equation*}
\Delta f=\frac{\frac{\sqrt{-1}}{2} \partial \overline{\partial} f \wedge \omega^{n-1}}{\omega^{n}}=g^{i \overline{j}} \frac{\partial^{2} f}{\partial z^{i} \partial \overline{z}^{j}}.
\end{equation*}
For convenience we write
\begin{equation*}
f_{i}=\frac{\partial f}{\partial z^{i}}, \quad f_{\overline{i}}=\frac{\partial f}{\partial \overline{z}^{i}}, \quad f_{i \overline{j}}=\frac{\partial^{2} f}{\partial z^{i} \partial \overline{z}^{j}}, \quad etc.
\end{equation*}

Let $\Lambda^{p, q}(M)$ denote the $(p,q)$-form on $M$. Then the exterior differential $d$ can be decomposed as $d=\partial+\overline{\partial}$, where
\begin{equation*}
\partial : \Lambda^{p, q} \rightarrow \Lambda^{p+1, q}, \overline{\partial} : \Lambda^{p, q} \rightarrow \Lambda^{p, q+1}.
\end{equation*}
By the Stokes theorem, we have
\begin{equation*}
\int_{M} \partial \alpha=\int_{\partial M} \alpha, \quad \forall \alpha \in \Lambda^{n-1, n}
\end{equation*}

\subsection{K\"ahler quotients}
Let $(M,\omega)$ be a symplectic manifold of dimension $n$. Assume there is a Lie group $G$ of dimension $k$ acting symplectically on $M$, i.e., $G$ acts by symplectomorphisms. Denote by $SV(M, \omega)$ the set of symplectic vector field on $M$ and $HV(M, \omega)$ the Hamiltonian vector field on $M$. Let $\mathcal{G}$ be the Lie algebra of $G$ and $\mathcal{G}^{*}$ its dual space. For any $Y \in \mathcal{G}$, it may generates a vector field $\tilde{Y}$ on $M$ defined by
\begin{equation*}
\tilde{Y}(p)=\frac{d}{dt} \Big|_{t=0} \operatorname{exp}(tY) \cdot p \in T_{p}M.
\end{equation*}
We can define the comoment map and the moment map as follow:
\begin{defn}
If there exist a Lie algebra anti-homeomorphism $\tilde{\mu}: \mathcal{G} \rightarrow C^{\infty}(M)$ such that
\begin{equation*}
i_{\tilde{Y}} \omega=-d\tilde{\mu}(Y), \quad \forall Y \in \mathcal{G},
\end{equation*}
then we say the action of $G$ on $(M,\omega)$ is Hamiltonian and $\tilde{\mu}$ is called the comoment map.\\
\indent Let $\mu: M \rightarrow \mathcal{G}^{*}$ such that for all $p \in M$ and $Y \in \mathcal{G}$, $<\mu(p),Y>$ is the Hamiltonian function of $\tilde{Y}$, i.e.
\begin{equation*}
<\mu(p),Y>=\tilde{\mu}(Y)(p).
\end{equation*}
Then $\mu$ is called the moment map.
\end{defn}
\noindent Note that $\mu$ is $G$-equivatiant, i.e.
\begin{equation*}
\mu(\beta \cdot p)=Ad_{\beta}^{*}(\mu(p)), \quad \forall p \in M, \forall \beta \in G.
\end{equation*}
The existence of a comoment map is equivalent to the existence of a $G$-equivatiant moment map.

In these circumstances, one can perform the so-called symplectic quotient. Let $G$ acts freely on $M$. Assume the action of $G$ is hamiltonian. Then for any $\tau$ a $G$-orbit in $\mathcal{G}^{*}$, we can define the symplectic quotient to be
\begin{equation*}
X_{\tau} :=\mu^{-1}(\tau) / G.
\end{equation*}
By the Marsden-Weinstein-Meyer Reduction Theorem \cite{M-W-M1} \cite{M-W-M2}, we have that $X_{\tau}$ is a symplectic manifold of dimension $n-2k$ with symplectic form $\omega_{\tau}$ if $\tau$ is a regular value for $\mu$. Let $\pi_{\tau}:\mu^{-1}(\tau) \rightarrow \mu^{-1}(\tau) / G$ be the natural projection and $i:\mu^{-1}(\tau) \hookrightarrow M$ the natural inclusion, then
\begin{equation*}
i^* \omega=\pi_{\tau}^* \omega_{\tau}.
\end{equation*}
The pair $(X_{\tau},\omega_{\tau})$ is called the symplectic quotient of $(M,\omega)$ with respect to $G, \mu$, or the symplectic reduction, or the reduced space, or the Marsden-Weinstein-Meyer quotient, etc.

If $(M, \omega)$ is a K\"ahler manifold and $G$ acting symplectically on $M$ via holomorphic ismetries. These results carry though to the complex structure of the K\"ahler quotients. We can think of $\mathcal{G}$ as a sub-bundle of $T\mu^{-1}(\tau)$. In fact, for any $Y \in \mathcal{G}$, we can prove that $\tilde{Y} \in T(\mu^{-1}(\tau))$. Let $Q_{p}(\tau) \subset T_p \mu^{-1}(\tau)$ be the orthogonal complement of $\mathcal{G}$ with respect to $\omega$. Hence we have the orthogonal decomposition of the tangent space
\begin{equation*}
T_{p} M=Q_{p}(\tau) \oplus \mathcal{G}_{p} \oplus J \mathcal{G}_{p},
\end{equation*}
where $J$ is the complex structure.
It is easy to check that $Q(\tau)$ is $J$-invariant and $d\pi_{\tau}:Q(\tau)\rightarrow TX_{\tau}$ induces an isomorphism.

The complex structure $J$ also induces a complex structure on the K\"ahler reduction $X_{\tau}$ defined by
\begin{equation*}
d \pi_{\tau} \circ J=J_{\tau} \circ d \pi_{\tau}.
\end{equation*}
Moreover, if we consider the direct sum decomposition of
\begin{equation*}
Q(\tau) \otimes \mathbb{C}=Q(\tau)^{(1,0)} \oplus Q(\tau)^{(0,1)},
\end{equation*}
then $d \pi_{\tau}$ induces an isomorphism
\begin{equation*}
Q^{(1,0)}(\tau) \rightarrow T^{(1,0)} X_{\tau}.
\end{equation*}
One can check that the induced complex structure $J_{\tau}$ on $X_{\tau}$ is integrable. One can also prove that the complex structure does not change as long as the moment map does not cross critical values.
We define the natural Riemannian metric $g_{\tau}$ on $X_{\tau}$ by
\begin{equation*}
g_{\tau}\left(d \pi_{\tau}\left(U\right), d \pi_{\tau}\left(W\right)\right)=g\left(U, W\right), \quad \text { for all } U, W \in Q(\tau).
\end{equation*}
Then the metric $g_{\tau}$ is compatible with $J_{\tau}$ and $g_{\tau}$ is in fact K\"ahler.

For simplicity, we assume that $G=S^1$ and its Lie algebra is identified with $\mathbb{R}$. When $G$ is a torus, any level is preserved and quotient at $\tau$ for the moment map $\mu$, is equivalent to quotient at $0$ for a shifted moment map $\varphi: M \rightarrow \mathcal{G}^{*}$, $\varphi:= \mu(p)-\tau$. Let $V$ be the vector field generating such an action. It follows that $\mu^{-1}(t_1) / S^1$ and $\mu^{-1}(t_2) / S^1$ are biholomorphic to each other whenever $t_1$ and $t_2$ are in an interval which does not contain any critical values of $\mu$. The following lemma is due to La Nave-Tian \cite{LT}.
\begin{lemma}[L-T]
If $V$ has no zeroes in a neighborhood of $\mu^{-1}([a,a+t_0])$, then the 1-parameter group of diffeomorphisms $\phi_{t}: M \rightarrow M$ generated by the vector field $U=\frac{JV}{|V|^2_g}$ induces biholomorphisms $\tilde{\phi_t}:X_a \rightarrow X_{t+a}$ for $t \in [0, t_0]$.
\end{lemma}
\noindent Then we get an one-parameter family of metrics on $X_a$, $h(\tau)=\phi_{\tau}^{*} g_{a+\tau}$, so long as there are no critical points of $\mu$ in $\{a+s b : 0 \leq s \leq 1\}$, where $\tau=s b$ and $g_{a+\tau}$ is the symplectic reduction of $g$ on $X_{a+\tau}$.

\section{Zero order estimate}

In this section, we give the proof of Propositon \ref{zero estimate negative lambda}.
First we recall a fact due to Zhu \cite{Zhu}

\begin{lemma}[Zhu]\label{Zhu's lem}
There is a uniform constant $C=C(\omega)$ such that for any $S^{1}$-invariant $u$ with $\omega+\frac{\sqrt{-1}}{2} \partial \overline{\partial} u \geq 0$, we have
\begin{equation}
|J V(u)| \leq C.
\end{equation}
\end{lemma}
We prove Propsiton \ref{zero estimate negative lambda} by proving a Cherrier-type inequality and the lemmas in \cite{TW}.
\begin{proof}[Proof of Propositon \ref{zero estimate negative lambda}]
First, we prove $\inf_{M}u\geq-C$. Let $p$ be the minimum point of $u$. Using Proposition \ref{Laplacian estimate negative lambda} and the maximum principle, at $p$, we obtain
\begin{equation*}
0 \leq \log\frac{(\omega+\frac{\sqrt{-1}}{2} \partial \overline{\partial} u)^{n}}{\omega^{n}} = \log(|Z|_{\ti{g}}^{2}+\ve)+F+u \leq C+\inf_{M}u,
\end{equation*}
which implies $\inf_{M}u\geq-C$.

Second, we prove $\sup_{M} u \leq C$. Due to Tosatti and Weinkove's results, it sufficient to prove a Cherrier-type inequality.
Integrating by parts, we get for $p \geq 1$
\begin{equation}
\begin{split}
\int_{M} e^{-p u}\left(\tilde{\omega}^{n}-\omega^{n}\right)={}& \frac{\sqrt{-1}}{2}\sum_{k=1}^{n} \int_{M} e^{-p u} \partial \bar{\partial} u \wedge \tilde{\omega}^{n-k} \wedge \omega^{k-1}\\[2mm]
\geq{}& \frac{\sqrt{-1}}{2} \int_{M} e^{-p u} \partial \overline{\partial} u \wedge \omega^{n-1}\\[2mm]
={}& \frac{\sqrt{-1}}{2} \int_{M}-\partial\left(e^{-p u}\right) \wedge \overline{\partial} u \wedge \omega^{n-1}\\[2mm]
={}& \frac{\sqrt{-1}}{2} p \int_{M} e^{-p u} \partial u \wedge \bar{\partial} u \wedge \omega^{n-1}\\[2mm]
={}& \frac{\sqrt{-1}}{2} p \int_{M}\left(e^{-\frac{p}{2} u} \partial u\right) \wedge\left(e^{-\frac{p}{2} u} \overline{\partial} u\right) \wedge \omega^{n-1}\\[2mm]
={}& \frac{4}{p} \int_{M}\left|\nabla e^{-\frac{p}{2} u}\right|_{g}^{2} \omega^{n}.\\[2mm]
\end{split}
\end{equation}
Multiplying $e^{-pu}$ on both sides of \eqref{scalar V-soliton equation, Z,perturbation} and integrating, we deduce from the above
\begin{equation}
\begin{split}
\frac{4}{p} \int_{M}\left|\nabla e^{-\frac{p}{2} u}\right|_{g}^{2} \omega^{n} \leq{}& \int_{M} e^{-p u}\left(\tilde{\omega}^{n}-\omega^{n}\right)\\[2mm]
={}& \int_{M} e^{-p u}\left[\left(|Z|_{\tilde{g}}^{2}+\varepsilon\right) e^{F+u}-1\right]\omega^{n}.
\\[2mm]
\end{split}
\end{equation}
Using Lemma 3.4 in \cite{LT}, we can compute
\begin{equation}\label{ddbaru}
\frac{\sqrt{-1}}{2} \partial \overline{\partial} u(V, J V)=\frac{1}{4} J V(J V(u)).
\end{equation}
It follows
\begin{equation}\label{gradient int estimate}
\begin{split}
\frac{4}{p} \int_{M}\left|\nabla e^{-\frac{p}{2} u}\right|_{g}^{2} \omega^{n} \leq{}& \int_{M}\left[C e^{-pu+F+u}+\frac{1}{4} JVJV(u) e^{-pu+F+u}-e^{-pu}\right]\omega^{n}\\[2mm]
\leq{}& C \int_{M} e^{-pu+F+u} \omega^{n}+\int_{M}\frac{1}{4} JVJV(u) e^{-pu+F+u}\omega^{n}.
\end{split}
\end{equation}
Using lemma \ref{Zhu's lem}, we have $Z(u)=JV(u)$ is real-valued and bounded. Recall the identity
\begin{equation}
\begin{split}
&\operatorname{div}\left(e^{-pu+F+u} Z(u)Z\right)\\[2mm]
={}& e^{-pu+F+u} Z(u) \operatorname{div}(Z)+Z\left(e^{-pu+F+u} Z(u)\right)\\[2mm]
={}& e^{-pu+F+u} Z(u) \operatorname{div}(Z)+e^{-p u+F+u} Z(-pu+F+u) Z(u)\\[2mm]
&+e^{-pu+F+u} Z(Z(u)),
\end{split}
\end{equation}
where the divergence is taken with respect to the metric $\omega$.
Therefore
\begin{equation}
e^{-pu+F+u} Z(Z(u)) \leq \operatorname{div}\left(e^{-pu+F+u} Z(u)Z\right)+Cpe^{-pu+F+u}.
\end{equation}
Plugging this into \eqref{gradient int estimate} and using lemma \ref{Zhu's lem}, we obtain
\begin{equation}\label{gradient int estimate 2}
\int_{M} | \nabla e^{-\frac{p}{2}u}|_{g}^{2} \omega^{n} \leq Cp^{2} \int_{M} e^{-u(p-1)} \omega^{n}.
\end{equation}
Now we can apply the standard Moser iteration scheme. Using \eqref{gradient int estimate 2} and the Sobolev inequality, we have for $\beta=\frac{n}{n-1}>1$,
\begin{equation}
\begin{split}
\left(\int_{M} e^{-p \beta u} \omega^{n}\right)^{1 / \beta} \leq &C\left(\int_{M}\left| \nabla e^{-\frac{p}{2} u}\right|^{2} \omega^{n}+\int_{M} e^{-p u} \omega^{n}\right)\\[2mm]
\leq & C p^{2}\left(\int_{M} e^{-u(p-1)}+e^{-pu} \omega^{n}\right)\\[2mm]
\leq & C p^{2} \int_{M} e^{-u(p-1)}\omega^{n},
\end{split}
\end{equation}
where we have used the fact that $u\geq -C$. Thus
\begin{equation}
\|e^{-u}\|_{L^{p{\beta}}} \leq C^{\frac{1}{p}} p^{\frac{2}{p}}\|e^{-u}\|_{L^{p-1}}^{\frac{p-1}{p}}.
\end{equation}
We can iterate this estimate in a standard way to obtain
\begin{equation}
\|e^{-u}\|_{L^{\infty}}\leq C\|e^{-u}\|_{L^{p_{0}-1}}^{\frac{p_{0}-1}{p_{0}+\frac{1}{\beta-1}}},
\end{equation}
which is equivalent to
\begin{equation}
e^{-\left(p_{0}+\frac{1}{\beta-1}\right) \inf_{M} u} \leq C \int_{M} e^{-(p_{0}-1)u}\omega^{n}.
\end{equation}
Then we have
\begin{equation}\label{ieration eu}
\begin{split}
C \int_{M} e^{-(p_{0}-1)u}\omega^{n} \geq{}& e^{-\left(p_{0}+\frac{1}{\beta-1}\right)\inf_{M}u}\\[2mm]
={}& e^{-(p_{0}-1) \inf_{M}u} \;\; e^{-\left(1+\frac{1}{\beta-1}\right) \inf_{M} u}\\[2mm]
\geq{}& C e^{-(p_{0}-1) \inf_{M}u},
\end{split}
\end{equation}
where we have used the condition $\inf_{M}u \leq C_{0}$.
Let $q=p_{0}-1$, then \eqref{ieration eu} becomes
\begin{equation}
e^{-q \inf_{M}u} \leq C \int_{M}e^{-qu}\omega^{n}.
\end{equation}

The $L^{\infty}$ bound on $u$ now follow from the arguments of \cite{TW}. More precisely, we can adapt the arguments in \cite{TW} to estimate $\sup_{M} u$.

Now we consider the case $|\operatorname{div}(Z)|^2-\operatorname{Ric}(Z,Z) \geq 0$.
By (\ref{scalar V-soliton equation, Z,perturbation}), we have
\begin{equation}\label{Infimum estimate equation 1}
\begin{split}
\vol(M,\omega)
={}& \int_{M}(\omega+\ddbar u)^{n} \\
={}& \int_{M}(|Z|_{\ti{g}}^{2}+\ve)e^{F+u}\omega^{n} \\
\geq {}& e^{-C+\inf_{M}u}\int_{M}|Z|_{\ti{g}}^{2}\omega^{n}.
\end{split}
\end{equation}
For the term $\int_{M}|Z|_{\ti{g}}^{2}\omega^{n}$, using \eqref{ddbaru} and the divergence theorem, we can compute
\begin{equation}\label{int|Z|^2_u}
\begin{split}
\int_{M}|Z|_{\ti{g}}^{2}\omega^{n}
= {} & \int_{M}|Z|_{g}^{2}\omega^{n} + \frac{1}{4} \int_{M} Z\ov{Z}(u) \omega^{n} \\
= {} & \int_{M}|Z|_{g}^{2}\omega^{n} - \frac{1}{4} \int_M \ov{Z}(u) \operatorname{div}(Z) \omega^n \\
= {} & \int_{M}|Z|_{g}^{2}\omega^{n} + \frac{1}{4} \int_M u \Big(\ov{Z} (\operatorname{div}(Z))+{|\operatorname{div}(Z)|}^2\Big) \omega^n.
\end{split}
\end{equation}
Choosing normal coordinates, we compute
\begin{equation}\label{Ric(Z,barZ)}
\begin{split}
\ov{Z} (\operatorname{div}(Z))
={}& \ov{Z}^{j} \partial_{\ov{j}}(\partial_{i}Z^i+\Gamma_{i k}^{i} Z^k)\\
={}& \ov{Z}^{j} (\partial_{\ov j} \Gamma_{i k}^{i}) Z^k\\
={}& Z^k \ov{Z}^{j} (-g^{i \ov l} \operatorname{R}_{i \ov{l} k \ov{j}})\\
={}& - \operatorname{Ric}(Z, \ov{Z}),
\end{split}
\end{equation}
where we used $Z$ is holomorphic and formulaes in Section 2.

Plugging \eqref{Ric(Z,barZ)} into \eqref{int|Z|^2_u}, we obtain
\begin{equation}
\int_{M}|Z|_{\ti{g}}^{2}\omega^{n}= \int_{M}|Z|_{g}^{2}\omega^{n} + \frac{1}{4} \int_M u \Big(- \operatorname{Ric}(Z, \ov{Z})+{|\operatorname{div}(Z)|}^2\Big) \omega^n.
\end{equation}
Without loss of generality, we assume that $\inf_{M} u \geq 1$. Combining this with $|\operatorname{div}(Z)|^2-\operatorname{Ric}(Z,Z) \geq 0$, we get
\begin{equation}
\int_{M}|Z|_{\ti{g}}^{2}\omega^{n} \geq \int_{M}|Z|_{g}^{2}\omega^{n}=C'^{-1}.
\end{equation}
Substituting this into \eqref{Infimum estimate equation 1}, we obtain that there is a uniform constant $C_{0}$ such that
\begin{equation}
\inf_{M} u \leq C_{0}.
\end{equation}
\end{proof}

\section{Second order estimate}

In this section, we give the proof of Proposition \ref{Laplacian estimate}, Propositon \ref{Laplacian estimate Positive Curvature} and Propositon \ref{Laplacian estimate negative lambda}.
We shall follow the notations in Section 2 and use ordinary derivatives.

First, we prove the following lemma.
\begin{lemma}\label{LT Lemma}
Let Z be any holomorphic vector field on a K\"ahler manifold (M,g). Then, for any $\ve>0$,
\begin{equation*}
\ddb\log(|Z|_{g}^{2}+\ve)+\frac{R(Z,\ov{Z},\cdot,\cdot)}{|Z|_{g}^{2}+\ve} > 0,
\end{equation*}
where $R$ is the curvature of $g$.
\end{lemma}

\begin{proof}
By direct calculation, in any normal coordinate system with respect to $\omega $ :
\begin{eqnarray*}
\partial _{\overline{j}} \partial _{i}(|Z|_{g}^{2}) &=&-R_{i\overline{j}k%
\overline{l}}Z^{k}\overline{Z^{l}}+g_{k\overline{l}}Z_{i}^{k}\overline{Z_{j}^{l}} \\
\partial _{\overline{j}} \partial _{i}(\log (|Z|_{g}^{2}+\varepsilon )) &=&%
\frac{\partial _{\overline{j}}\partial _{i}(|Z|_{g}^{2})}{%
|Z|_{g}^{2}+\varepsilon }-\frac{\partial _{i}(|Z|_{g}^{2})\partial _{%
\overline{j}}(|Z|_{g}^{2})}{(|Z|_{g}^{2}+\varepsilon )^{2}} \\
&=&-\frac{Rm(Z,\overline{Z},\partial _{i},\partial _{\overline{j}})}{%
|Z|_{g}^{2}+\varepsilon }+\frac{g_{k\overline{l}}Z_{i}^{k}\overline{Z_{j}^{l}%
}}{|Z|_{g}^{2}+\varepsilon }-\frac{g_{k\overline{l}}Z_{i}^{k}\overline{Z^{l}}%
g_{s\overline{t}}Z^{s}\overline{Z_{j}^{t}}}{(|Z|_{g}^{2}+\varepsilon )^{2}}.
\end{eqnarray*}

We need to show that: $\forall \xi ^{i},\xi ^{j}\in \mathbb{C},$%
\begin{equation*}
K:=\xi ^{i}\overline{\xi ^{j}}(\frac{g_{k\overline{l}}Z_{i}^{k}\overline{%
Z_{j}^{l}}}{|Z|_{g}^{2}+\varepsilon }-\frac{g_{k\overline{l}}Z_{i}^{k}%
\overline{Z^{l}}g_{s\overline{t}}Z^{s}\overline{Z_{j}^{t}}}{%
(|Z|_{g}^{2}+\varepsilon )^{2}})>0.
\end{equation*}
In fact,%
\begin{eqnarray*}
K &=&\frac{\sum_{i,j,k}\xi ^{i}Z_{i}^{k}\overline{\xi ^{j}}\overline{%
Z_{j}^{k}}}{|Z|_{g}^{2}+\varepsilon }-\frac{\sum_{i,j,k,s}\xi ^{i}Z_{i}^{k}%
\overline{Z^{k}}\overline{\xi ^{j}}Z^{s}\overline{Z_{j}^{s}}}{%
(|Z|_{g}^{2}+\varepsilon )^{2}} \\
&=&\frac{\sum_{k}|\sum_{i}\xi ^{i}Z_{i}^{k}|^{2}}{|Z|_{g}^{2}+\varepsilon }-%
\frac{|\sum_{i,k}\xi ^{i}Z_{i}^{k}\overline{Z^{k}}|^{2}}{(|Z|_{g}^{2}+%
\varepsilon )^{2}} \\
&\geq &\frac{\sum_{k}|\sum_{i}\xi ^{i}Z_{i}^{k}|^{2}}{|Z|_{g}^{2}+%
\varepsilon }-\frac{(\sum_{k}|\overline{Z^{k}}|^{2})(\sum_{k}|\sum_{i}\xi
^{i}Z_{i}^{k}|^{2})}{(|Z|_{g}^{2}+\varepsilon )^{2}} \\
&>&0.
\end{eqnarray*}
\end{proof}
\noindent We will make ubiquitous use of this lemma.

Now we are in a position to prove Proposition \ref{Laplacian estimate}.

\begin{proof}[Proof of Proposition \protect\ref{Laplacian estimate}]
We consider the following quantity
\begin{equation*}
Q=\log \text{tr}_{\omega }\tilde{\omega}-A u ,
\end{equation*}%
where $A$ is a constant to be determined later. Let $p$ be the maximum point
of $Q$. We choose the normal coordinate system (with respect to $\omega $)
centered at $p$ such that
\begin{eqnarray*}
g_{i \bar{j}}(p)=\delta_{i j} \quad  \text { and }& \quad dg(p)=0,\\
\tilde{g}_{i \bar{j}}(p)=\delta_{i j} \tilde{g}_{i \bar{i}}(p) \quad  \text { and }& \quad \tilde{g}_{1 \overline{1}}(p) \geqslant \tilde{g}_{1 \overline{1}}(p) \geqslant \cdots \geqslant \tilde{g}_{n \bar{n}}(p).
\end{eqnarray*}

For convenience, we use the following notation:
\begin{equation*}
\tilde{g}_{H,\varepsilon }^{i\overline{j}}=\tilde{g}^{i\overline{j}}
-\frac{Z^{i}\overline{Z}^{j}}{|Z|_{\tilde{g}}^{2}+\varepsilon }.
\end{equation*}
Define the linearized operator by
\begin{equation*}
\Delta _{\tilde{g}_{H,\varepsilon }}=\tilde{g}_{H,\varepsilon }^{i\overline{j}} \partial_{\ov{j}} \partial_i
\end{equation*}

We compute
\begin{equation}
\begin{split}
& \Delta _{\tilde{g}_{H,\varepsilon }}(\text{tr}_{\omega }\tilde{\omega}) \\%
[2mm]
={}& g^{k\overline{l}}\tilde{g}_{H,\varepsilon }^{i\overline{j}}\partial
_{\overline{j}}\partial _{i}\tilde{g}_{k\overline{l}}+\sum_{i,j,k,l}\tilde{g}%
_{H,\varepsilon }^{i\overline{j}}R_{i\overline{j}k\overline{l}}\tilde{g}_{k%
\overline{l}} \\[2mm]
={}& g^{k\overline{l}}\tilde{g}_{H,\varepsilon }^{i\overline{j}}(-\tilde{R}%
_{i\overline{j}k\overline{l}}+\tilde{g}^{s\overline{t}}\partial _{i}\tilde{g}%
_{s\overline{l}}\partial _{\overline{j}}\tilde{g}_{k\overline{t}%
})+\sum_{i,j,k,l}\tilde{g}_{H,\varepsilon }^{i\overline{j}}R_{i\overline{j}k%
\overline{l}}\tilde{g}_{k\overline{l}} \\[2mm]
={}& -g^{k\overline{l}}\tilde{g}_{H,\varepsilon }^{i\overline{j}}\tilde{R}_{i%
\overline{j}k\overline{l}}+g^{k\overline{l}}\tilde{g}_{H,\varepsilon }^{i%
\overline{j}}\tilde{g}^{s\overline{t}}\partial _{i}\tilde{g}_{s\overline{l}%
}\partial _{\overline{j}}\tilde{g}_{k\overline{t}}+\sum_{i,j,k,l}\tilde{g}%
_{H,\varepsilon }^{i\overline{j}}R_{i\overline{j}k\overline{l}}\tilde{g}_{k%
\overline{l}} \\
={}& -g^{k\overline{l}}\tilde{R}_{k\overline{l}}+g^{k\overline{l}}\frac{Z^{i}%
\overline{Z^{j}}}{|Z|_{\tilde{g}}^{2}+\varepsilon }\tilde{R}_{i\overline{j}k%
\overline{l}}+g^{k\overline{l}}\tilde{g}_{H,\varepsilon }^{i\overline{j}}%
\tilde{g}^{s\overline{t}} u _{s\overline{l}i} u _{k\overline{t}%
\overline{j}}+\sum_{i,j,k,l}\tilde{g}_{H,\varepsilon }^{i\overline{j}}R_{i%
\overline{j}k\overline{l}}\tilde{g}_{k\overline{l}}.
\end{split}
\label{Laplacian estimate equation 1}
\end{equation}%
Differentiating (\ref{scalar V-soliton equation, Z,perturbation}) twice, we obtain
\begin{equation*}
-\tilde{R}_{k\overline{l}}=\partial _{\overline{l}} \partial _{k} \log (|Z|_{\tilde{g}}^{2}+\varepsilon)+F_{k\overline{l}}-\lambda u_{k\overline{l}}-R_{k\overline{l}}.
\end{equation*}%
Substituting this into (\ref{Laplacian estimate equation 1}) and using Lemma %
\ref{LT Lemma}, it follows that%
\begin{eqnarray}
\Delta _{\tilde{g}_{H,\varepsilon }}(\text{tr}_{\omega }\tilde{\omega})
&\geq &{}\Delta (\log (|Z|_{\tilde{g}}^{2}+\varepsilon ))+g^{k\overline{l}}%
\frac{Z^{i}\overline{Z}^{j}}{|Z|_{\tilde{g}}^{2}+\varepsilon }\tilde{R}_{i%
\overline{j}k\overline{l}}+\sum_{i,j,s,k}\tilde{g}_{H,\varepsilon }^{i%
\overline{j}}\tilde{g}^{s\overline{s}} u _{s\overline{k}i} u _{k%
\overline{s}\overline{j}}  \notag \\
&&+\sum_{i,j,k,l}\tilde{g}_{H,\varepsilon }^{i\overline{j}}R_{i%
\overline{j}k\overline{l}}\tilde{g}_{k\overline{l}}+\Delta F-\lambda \Delta u-R  \notag \\
&\geq &{}\sum_{i,j,s,k}\tilde{g}_{H,\varepsilon }^{i\overline{j}}\tilde{g}^{s%
\overline{s}} u _{s\overline{k}i} u _{k\overline{s}\overline{j}%
}+\sum_{i,j,k,l}\tilde{g}_{H,\varepsilon }^{i\overline{j}}R_{i\overline{j}l%
\overline{k}}\tilde{g}_{k\overline{l}}-C.
\label{Laplacian estimate equation 11}
\end{eqnarray}
This implies
\begin{equation}
\begin{split}
\Delta _{\tilde{g}_{H,\varepsilon }}(\log (\text{tr}_{\omega }\tilde{\omega}%
))={}& \frac{\Delta _{\tilde{g}_{H,\varepsilon }}(\text{tr}_{\omega }\tilde{%
\omega})}{\text{tr}_{\omega }\tilde{\omega}}-\frac{\sum_{i,j,k,l}\tilde{g}%
_{H,\varepsilon }^{i\overline{j}} u _{k\overline{k}i} u _{l%
\overline{l}\overline{j}}}{(\text{tr}_{\omega }\tilde{\omega})^{2}} \\
\geq {}& \frac{\sum_{i,j,s,k}\tilde{g}_{H,\varepsilon }^{i\overline{j}}%
\tilde{g}^{s\overline{s}} u _{s\overline{k}i} u _{k\overline{s}%
\overline{j}}}{\text{tr}_{\omega }\tilde{\omega}}-\frac{\sum_{i,j,k,l}\tilde{%
g}_{H,\varepsilon }^{i\overline{j}} u _{k\overline{k}i} u _{l%
\overline{l}\overline{j}}}{(\text{tr}_{\omega }\tilde{\omega})^{2}} \\
& +\frac{\sum_{i,j,k,l}\tilde{g}_{H,\varepsilon }^{i\overline{j}}R_{i%
\overline{j}k\overline{l}}\tilde{g}_{k\overline{l}}}{\text{tr}_{\omega }%
\tilde{\omega}}-C,
\end{split}
\label{Laplacian estimate equation 2}
\end{equation}%
where we assume $\text{tr}_{\omega }\tilde{\omega}>1$ without loss of
generality.

First, we need to deal with the third order terms in (\ref%
{Laplacian estimate equation 2}). Since $\tilde{g}_{H,\varepsilon }^{i%
\overline{j}}$ is a positive matrix, there exists a Hermitian matrix $B$
such that
\begin{equation*}
\tilde{g}_{H,\varepsilon }^{i\overline{j}}=\sum_{s}B^{i\overline{s}}%
\overline{B^{j\overline{s}}}.
\end{equation*}%
We compute

\begin{equation*}
\begin{split}
\sum_{i,j,k,l}\tilde{g}_{H,\varepsilon }^{i\overline{j}} u _{k\overline{%
k}i} u _{l\overline{l}\overline{j}}={}& \sum_{s}|\sum_{i,k}B^{i%
\overline{s}} u _{k\overline{k}i}|^{2} \\
={}& \sum_{s}|\sum_{k}\sqrt{\tilde{g}_{k\overline{k}}}\sum_{i}\sqrt{\tilde{g}%
^{k\overline{k}}}B^{i\overline{s}} u _{k\overline{k}i}|^{2} \\
\leq {}& \text{tr}_{\omega }\tilde{\omega}\sum_{k,s}\tilde{g}^{k\overline{k}%
}|\sum_{i}B^{i\overline{s}} u _{k\overline{k}i}|^{2} \\
\leq {}& \text{tr}_{\omega }\tilde{\omega}\sum_{k,l,s}\tilde{g}^{k\overline{k%
}}|\sum_{i}B^{i\overline{s}} u _{k\overline{l}i}|^{2} \\
={}& \text{tr}_{\omega }\tilde{\omega}\sum_{i,j,k,l,s}\tilde{g}^{k\overline{k%
}}B^{i\overline{s}} u _{k\overline{l}i}\overline{B^{j\overline{s}}}%
 u _{l\overline{k}\overline{j}} \\
={}& \text{tr}_{\omega }\tilde{\omega}\sum_{i,j,k,l}\tilde{g}^{k\overline{k}}%
\tilde{g}_{H,\varepsilon }^{i\overline{j}} u _{k\overline{l}i} u
_{l\overline{k}\overline{j}},
\end{split}%
\end{equation*}
which implies
\begin{equation}
\frac{\sum_{i,j,s,k}\tilde{g}_{H,\varepsilon }^{i\overline{j}}\tilde{g}^{s%
\overline{s}} u _{s\overline{k}i} u _{k\overline{s}\overline{j}}}{%
\text{tr}_{\omega }\tilde{\omega}}-\frac{\sum_{i,j,k,l}\tilde{g}%
_{H,\varepsilon }^{i\overline{j}} u _{k\overline{k}i} u _{l%
\overline{l}\overline{j}}}{(\text{tr}_{\omega }\tilde{\omega})^{2}}>0.
\label{Laplacian estimate equation 3}
\end{equation}%

Next, we need to deal with the second order terms in (\ref{Laplacian
estimate equation 2}). We compute
\begin{eqnarray*}
&&\sum_{i,j,k,l}\tilde{g}_{H,\varepsilon }^{i\overline{j}}R_{i\overline{j}k%
\overline{l}}\tilde{g}_{k\overline{l}} \\
&=&\sum_{i,j,k,r}B^{i\overline{r}}\overline{B^{j\overline{r}}}R(\partial
_{i},\partial _{\overline{j}},\partial _{k},\partial _{\overline{k}})\tilde{g%
}_{k\overline{k}} \\
&=&\sum_{i,j,k,r}R(B^{i\overline{r}}\partial _{i},\overline{B^{j\overline{r}%
}\partial _{j}},\sqrt{\tilde{g}_{k\overline{k}}}\partial _{k},\overline{%
\sqrt{\tilde{g}_{k\overline{k}}}\partial _{k}}) \\
&\geq &-C\sum_{k,r}|\sum_{i} B^{i\overline{r}}\partial _{i}|^{2}|\sqrt{\tilde{g}_{k%
\overline{k}}}\partial _{k}|^{2} \\
&=&-C\sum_{k,r}\sum_{i}|B^{i\overline{r}}|^{2}|\sqrt{\tilde{g}_{k\overline{k}%
}}|^{2} \\
&=&-C\sum_{i}\tilde{g}_{H,\varepsilon }^{i\overline{i}}\text{tr}_{\omega }%
\tilde{\omega},
\end{eqnarray*}
which implies
\begin{equation}
\frac{\sum_{i,j,k,l}\tilde{g}_{H,\varepsilon }^{i\overline{j}}R_{i\overline{j%
}k\overline{l}}\tilde{g}_{k\overline{l}}}{\text{tr}_{\omega }\tilde{\omega}}%
\geq -C\sum_{i}\tilde{g}_{H,\varepsilon }^{i\overline{i}}.
\label{Laplacian estimate equation 8}
\end{equation}
Substituting (\ref{Laplacian estimate equation 3}) and (\ref{Laplacian
estimate equation 8}) into (\ref{Laplacian estimate equation 2}), we obtain
\begin{equation*}
\Delta _{\tilde{g}_{H,\varepsilon }}(\log \text{tr}_{\omega }\tilde{\omega}%
)\geq -C_{0}\sum_{i}\tilde{g}_{H,\varepsilon }^{i\overline{i}}-C,
\end{equation*}
where $C_{0}$, $C$ are uniform constants.

It is clear that
\begin{eqnarray*}
\Delta _{\tilde{g}_{H,\varepsilon }} u &=&\tilde{g}_{H,\varepsilon }^{i%
\overline{j}} u _{i\overline{j}} \\
&=&(\tilde{g}^{i\overline{j}}-\frac{Z^{i}\overline{Z}^{j}}{|Z|_{\tilde{g}%
}^{2}+\varepsilon })\tilde{g}_{i\overline{j}}-\tilde{g}_{H,\varepsilon }^{i%
\overline{j}}g_{i\overline{j}} \\
&=&n-\frac{|Z|_{\tilde{g}}^{2}}{|Z|_{\tilde{g}}^{2}+\varepsilon }-\sum_{i}%
\tilde{g}_{H,\varepsilon }^{i\overline{i}} \\
&\leq &n-\sum_{i}\tilde{g}_{H,\varepsilon }^{i\overline{i}}.
\end{eqnarray*}
Hence, after choosing $A=C_{0}+1$, we have
\begin{equation*}
\Delta _{\tilde{g}_{H,\varepsilon }}(\log \text{tr}_{\omega }\tilde{\omega}%
-A u )\geq \sum_{i}\tilde{g}_{H,\varepsilon }^{i\overline{i}}-C.
\end{equation*}
Applying the maximum principle, we see that
\begin{equation}
\sum_{i}\tilde{g}^{i\overline{i}}-\frac{|Z|_{g}^{2}}{|Z|_{\tilde{g}%
}^{2}+\varepsilon }=\sum_{i}\tilde{g}_{H,\varepsilon }^{i\overline{i}}\leq C,
\label{Laplacian estimate equation 4}
\end{equation}
which implies
\begin{equation*}
\sum_{i}\tilde{g}^{i\overline{i}}
\leq \frac{|Z|_{g}^{2}}{|Z|_{\tilde{g}}^{2}+\varepsilon }+C
\leq \frac{|Z|_{g}^{2}}{\tilde{g}_{i\overline{i}%
}|Z^{i}|^{2}+\ve}+C
\leq \frac{|Z|_{g}^{2}}{\tilde{g}_{n\overline{n}}|Z|_{g}^{2}+\ve}+C
=\tilde{g}^{n\overline{n}}+C.
\end{equation*}
Thus, for $i\leq n-1$, we obtain
\begin{equation}
\tilde{g}^{i\overline{i}}\leq C.  \label{Laplacian estimate equation 7}
\end{equation}

Recalling (\ref{scalar V-soliton equation, Z,perturbation}), it is clear that
\begin{equation}
\prod_{i=1}^{n}\tilde{g}_{i\overline{i}}\leq C(|Z|_{\tilde{g}%
}^{2}+\varepsilon ).  \label{Laplacian estimate equation 6}
\end{equation}%
Using (\ref{Laplacian estimate equation 4}) and (\ref{Laplacian estimate
equation 6}), we obtain
\begin{equation}
\prod_{i=1}^{n-1}\tilde{g}_{i\overline{i}}=\tilde{g}^{n\overline{n}%
}\prod_{i=1}^{n}\tilde{g}_{i\overline{i}}\leq \sum_{i=1}^{n}\tilde{g}^{i%
\overline{i}}C(|Z|_{\tilde{g}}^{2}+\varepsilon )\leq C|Z|_{\tilde{g}}^{2}+C.
\label{Laplacian estimate equation 9}
\end{equation}%
Combining this with (\ref{Laplacian estimate equation 7}), we have
\begin{equation*}
\tilde{g}_{1\overline{1}}=\frac{\prod_{i=1}^{n-1}\tilde{g}_{i\overline{i}}}{%
\prod_{i=2}^{n-1}\tilde{g}_{i\overline{i}}}\leq C|Z|_{\tilde{g}}^{2}+C,
\end{equation*}%
which implies
\begin{equation*}
n+\Delta  u \leq n\tilde{g}_{1\overline{1}}\leq C|Z|_{\tilde{g}}^{2}+C,
\end{equation*}%
as required.
\end{proof}

Now we are in a position to prove Proposition \ref{Laplacian estimate
Positive Curvature}.

\begin{proof}[Proof of Proposition \protect\ref{Laplacian estimate
Positive Curvature}]
Suppose the holomorphic bisectional curvature $\mathrm{BK}\geq D>0$.
Let $q$ be the maximum point of $\text{tr}_{\omega }\tilde\omega$.
We choose the normal coordinate system (with respect to $\omega $) centered
at $q$ such that
\begin{equation*}
\tilde{g}_{i\overline{j}}(q)=\delta _{ij}\tilde{g}_{i\overline{i}}(q)~\text{%
~and~}~\tilde{g}_{1\overline{1}}(q)\geq \tilde{g}_{1\overline{1}}(q)\geq
\cdots \geq \tilde{g}_{n\overline{n}}(q).
\end{equation*}

Recalling (\ref{Laplacian estimate equation 11}), we have
\begin{eqnarray*}
\Delta _{\tilde{g}_{H,\varepsilon }}(\text{tr}_{\omega }\tilde{\omega})
&\geq &{}{}\sum_{i,j,s,k}\tilde{g}_{H,\varepsilon }^{i\overline{j}}\tilde{g}%
^{s\overline{s}} u _{s\overline{k}i} u _{k\overline{s}\overline{j}%
}+\sum_{i,j,k,l}\tilde{g}_{H,\varepsilon }^{i\overline{j}}R_{i\overline{j}l%
\overline{k}}\tilde{g}_{k\overline{l}}-C \\
&\geq &\sum_{i,j,k,l}\tilde{g}_{H,\varepsilon }^{i\overline{j}}R_{i\overline{%
j}l\overline{k}}\tilde{g}_{k\overline{l}}-C.
\end{eqnarray*}
Similar to the process of deriving (\ref{Laplacian estimate equation 8}) in Proposition \ref{Laplacian estimate}, we can write
\begin{equation*}
\tilde{g}_{H,\varepsilon }^{i\overline{j}}=\sum_{s}B^{i\overline{s}}%
\overline{B^{j\overline{s}}}.
\end{equation*}
Then we can estimate
\begin{eqnarray*}
&&\sum_{i,j,k,l}\tilde{g}_{H,\varepsilon }^{i\overline{j}}R_{i\overline{j}k%
\overline{l}}\tilde{g}_{k\overline{l}} \\
&=&\sum_{i,j,k,r}B^{i\overline{r}}\overline{B^{j\overline{r}}}R(\partial
_{i},\partial _{\overline{j}},\partial _{k},\partial _{\overline{k}})\tilde{g%
}_{k\overline{k}} \\
&=&\sum_{i,j,k,r}R(B^{i\overline{r}}\partial _{i},\overline{B^{j\overline{r}%
}\partial _{j}},\sqrt{\tilde{g}_{k\overline{k}}}\partial _{k},\overline{%
\sqrt{\tilde{g}_{k\overline{k}}}\partial _{k}}) \\
&\geq &D\sum_{k,r}|\sum_{i} B^{i\overline{r}}\partial _{i}|^{2}|\sqrt{\tilde{g}_{k%
\overline{k}}}\partial _{k}|^{2} \\
&=&D\sum_{k,r}\sum_{i}|B^{i\overline{r}}|^{2}|\sqrt{\tilde{g}_{k\overline{k}%
}}|^{2} \\
&=&D\sum_{i}\tilde{g}_{H,\varepsilon }^{i\overline{i}}\text{tr}_{\omega }%
\tilde{\omega},
\end{eqnarray*}
which implies
\begin{equation*}
\Delta _{\tilde{g}_{H,\varepsilon }}(\text{tr}_{\omega }\tilde{\omega})\geq
D\sum_{i}\tilde{g}_{H,\varepsilon }^{i\overline{i}}(n+\Delta  u )-C.
\end{equation*}
At $q$, applying the maximum principle, we see that%
\begin{equation}
\sum_{i}\tilde{g}_{H,\varepsilon }^{i\overline{i}}(n+\Delta  u )\leq C.
\label{Laplacian estimate equation 10}
\end{equation}

We split up into different cases. Let B be a constant to be determined later.

\begin{case}
$\sum_{i}\tilde{g}_{H,\varepsilon }^{i\overline{i}}\geq \frac{1}{B}.$
\end{case}

This is a easy case. From (\ref{Laplacian estimate equation 10}), we get $%
(n+\Delta  u )\leq BC.$

\begin{case}
$\sum_{i}\tilde{g}_{H,\varepsilon }^{i\overline{i}}\leq \frac{1}{B}.$
\end{case}

In this case, we have
\begin{equation}
\sum_{i}\tilde{g}^{i\overline{i}}-\frac{|Z|_{g}^{2}}{|Z|_{\tilde{g}}^{2}
+\varepsilon }\leq \frac{1}{B},
\end{equation}
which implies
\begin{equation*}
\sum_{i}\tilde{g}^{i\overline{i}}\leq \frac{|Z|_{g}^{2}}{|Z|_{\tilde{g}}^{2}+\varepsilon }+\frac{1}{B}\leq \tilde{g}^{n\overline{n}}+\frac{1}{B}.
\end{equation*}%
Thus, for $i\leq n-1$, we obtain
\begin{equation}
\tilde{g}^{i\overline{i}}\leq \frac{1}{B}.
\end{equation}

Recalling (\ref{scalar V-soliton equation, Z,perturbation}), it is clear that
\begin{equation}
\prod_{i=1}^{n}\tilde{g}_{i\overline{i}}\leq C(|Z|_{\tilde{g}%
}^{2}+\varepsilon ).
\end{equation}
We obtain
\begin{equation}
\prod_{i=1}^{n-1}\tilde{g}_{i\overline{i}}=\tilde{g}^{n\overline{n}%
}\prod_{i=1}^{n}\tilde{g}_{i\overline{i}}\leq \sum_{i=1}^{n}\tilde{g}^{i%
\overline{i}}C(|Z|_{\tilde{g}}^{2}+\varepsilon )\leq \frac{C}{B}|Z|_{\tilde{g}}^{2}+C.
\end{equation}
Then we have
\begin{equation*}
\frac{1}{B}\geq \sum_{i=1}^{n-1}\tilde{g}^{i\overline{i}}=\sum_{i=1}^{n-1}%
\frac{1}{\tilde{g}_{i\overline{i}}}\geq (\frac{\sum_{i=1}^{n-1}\tilde{g}_{i%
\overline{i}}}{\prod\nolimits_{i=1}^{n-1}\tilde{g}_{i\overline{i}}})^{\frac{1%
}{n-2}}\geq (\frac{\sum_{i=1}^{n-1}\tilde{g}_{i\overline{i}}}{\frac{C}{B}%
|Z|_{\tilde{g}}^{2}+C})^{\frac{1}{n-2}}.
\end{equation*}
Thus%
\begin{eqnarray*}
\tilde{g}_{1\overline{1}} &\leq &\sum_{i=1}^{n-1}\tilde{g}_{i\overline{i}%
}\leq \frac{C}{B^{n-1}}|Z|_{\tilde{g}}^{2}+\frac{C}{B^{n-2}} \\
&\leq &\frac{C}{B^{n-1}}|Z|_{g}^{2}\tilde{g}_{1\overline{1}}+\frac{C}{B^{n-2}}.
\end{eqnarray*}
Choosing
\begin{equation*}
B=(2C\sup_{M}|Z|_{g}^{2})^{\frac{1}{n-1}},
\end{equation*}
we obtain $\tilde{g}_{1\overline{1}}\leq C$, which implies $(n+\Delta
 u )\leq C.$
\end{proof}

Now we give the proof of Proposition \ref{Laplacian estimate negative lambda}

\begin{proof}[Proof of Proposition \protect\ref{Laplacian estimate negative lambda}]
When $\lambda =-1$, equation \eqref{scalar V-soliton equation, Z,perturbation} becomes
\begin{equation}  \label{V-soliton equation nagetive lambda}
(\omega+\sqrt{-1} \partial \overline{\partial} u)^{n}=(|Z|_{\tilde{g}%
}^{2}+\varepsilon)e^{F+u}\omega^{n}.
\end{equation}
Let $p$ be the maximum point of $|Z|_{\ti{g}}^{2}$. Around $p$, we choose holomorphic coordinates such that
\begin{equation*}
\ti{g}_{i\ov{j}}(p)=\delta_{ij} \text{~and~} d\ti{g}_{i\ov{j}}(p)=0.
\end{equation*}
We compute at p,
\begin{equation*}
\begin{split}
\Delta_{\ti{g}_{H,\ve}}(|Z|_{\ti{g}}^{2})
= {} & \ti{g}_{H,\ve}^{i\ov{j}}Z^{k}\ov{Z}^{l}\de_{\ov{j}}\de_{i}\ti{g}_{k\ov{l}}
+\ti{g}_{H,\ve}^{i\ov{j}}\ti{g}_{k\ov{l}}Z_{i}^{k}\ov{Z}_{j}^{l} \\
\geq {} & -\left(\ti{g}^{i\ov{j}}-\frac{Z^{i}\ov{Z}^{j}}{|Z|_{\ti{g}}^{2}+\ve}\right)
\ti{R}_{i\ov{j}k\ov{l}}Z^{k}\ov{Z}^{l} \\
= {} & -\ti{R}_{k\ov{l}}Z^{k}\ov{Z}^{l}+
\frac{Z^{k}\ov{Z}^{l}Z^{i}\ov{Z}^{j}\ti{R}_{i\ov{j}k\ov{l}}}{|Z|_{\ti{g}}^{2}+\ve}.
\end{split}
\end{equation*}
Differentiating \eqref{V-soliton equation nagetive lambda} twice, we see that
\begin{equation*}
\begin{split}
& -\ti{R}_{k\ov{l}}Z^{k}\ov{Z}^{l} \\[1.5mm]
= {} & Z^{k}\ov{Z}^{l}\de_{\ov{l}}\de_{k}\log(|Z|_{\ti{g}}^{2}+\ve)
+Z^{k}\ov{Z}^{l}F_{k\ov{l}}
+Z^{k}\ov{Z}^{l}u_{k\ov{l}}-R_{k\ov{l}}Z^{k}\ov{Z}^{l}. \\[1.5mm]
\end{split}
\end{equation*}
Using Lemma \ref{LT Lemma}, it follows that
\begin{equation*}
\Delta_{\ti{g}_{H,\ve}}(|Z|_{\ti{g}}^{2}) \geq |Z|_{\ti{g}}^{2}-C.
\end{equation*}
Applying the maximum principle, we obtain the desired estimate.
\end{proof}

\end{document}